\documentclass[11pt,letterpaper]{article}
\usepackage{amsmath,amssymb,amsfonts,amsthm,mathrsfs, url}
\usepackage{graphicx}
\usepackage[usenames,dvipsnames]{color}

\usepackage[margin=1in]{geometry}

\usepackage[normalem]{ulem}

\theoremstyle{plain}
\newtheorem{thm}{Theorem}
\newtheorem{lem}[thm]{Lemma}
\newtheorem{cor}[thm]{Corollary}

\newtheorem{remark}[thm]{Remark}

\theoremstyle{definition}
\newtheorem{defn}{Definition}
\theoremstyle{remark}
\newtheorem{claim}{Claim}

\newcommand{\real}{\ensuremath {\mathbb R} }	

\newcommand{\nat}{\ensuremath {\mathbb N} }

\newcommand{\mbf}[1] {\text{\boldmath$#1$}}
\newcommand{\remove}[1] {}

\newcommand{\pr} {{\bf Pr}}

\newcommand{\cB} {\ensuremath{\mathcal B}}
\newcommand{\cC} {\ensuremath{\mathcal C}}
\newcommand{\cD} {\ensuremath{\mathcal D}}
\newcommand{\cG} {\ensuremath{\mathcal G}}
\newcommand{\cK} {\ensuremath{\mathcal K}}
\newcommand{\cN} {\ensuremath{\mathcal N}}
\newcommand{\cR} {\ensuremath{\mathcal R}}
\newcommand{\cS} {\ensuremath{\mathcal S}}
\newcommand{\cU} {\ensuremath{\mathcal U}}
\newcommand{\sG} {\ensuremath{\mathscr G}}
\newcommand{\bX} {\ensuremath{\mbf X}}

\newcommand{\bZ} {\ensuremath{\mbf Z}}
\newcommand{\fP} {\ensuremath{\mathfrak P}}
\newcommand{\eps}{\varepsilon}
\newcommand{\bfrac}[2]{\of{\frac{#1}{#2}}}
\newcommand{\of}[1]{\left( #1 \right)}

\DeclareMathOperator{\polylog}{polylog}

\DeclareMathOperator{\vol}{vol}

\DeclareMathOperator{\1conn}{1-conn}
\DeclareMathOperator{\2conn}{2-conn}
\DeclareMathOperator{\kconn}{\mathit{k}-conn}
\DeclareMathOperator{\RPM}{RPM}
\DeclareMathOperator{\RHC}{RHC}

\title{Rainbow perfect matchings and Hamilton cycles in the random geometric graph}

\author{Deepak Bal\footnote{Department of Mathematical Sciences, Montclair State University, Montclair, NJ, 07043 U.S.A. \texttt{deepak.bal@montclair.edu}}\and  Patrick Bennett\thanks{Department of Mathematics, Western Michigan University, Kalamazoo, MI, 49008 U.S.A.
\texttt{patrick.bennett@wmich.edu}}
\and Xavier P\'erez-Gim\'enez\thanks{Department of Mathematics, Ryerson University, Toronto, ON, Canada, M5B 2K3.
\texttt{xperez@ryerson.ca}}
\and Pawe\l{} Pra\l{}at\thanks{Department of Mathematics, Ryerson University, Toronto, ON, Canada, M5B 2K3.
Research is supported in part by NSERC and Ryerson University.
\texttt{pralat@ryerson.ca}} }

\date{}

\begin{document}
\maketitle

\begin{abstract}
Given a graph on $n$ vertices and an assignment of colours to the edges, a rainbow Hamilton cycle is a cycle of length $n$ visiting each vertex once and with pairwise different colours on the edges.
Similarly (for even $n$) a rainbow perfect matching is a collection of $n/2$ independent edges with pairwise different colours.
In this note we show that if we randomly colour the edges of a random geometric graph
with sufficiently many colours, then a.a.s.\ 
the graph contains a rainbow perfect matching (rainbow Hamilton cycle) if and only if the minimum degree is at least $1$ (respectively, at least $2$).
More precisely, consider $n$ points (i.e.~vertices) chosen independently and uniformly at random from the unit $d$-dimensional cube for any fixed $d\ge2$.  Form a sequence of graphs on these $n$ vertices by adding edges one by one between each possible pair of vertices. Edges are added in increasing order of lengths (measured with respect to the $\ell_p$ norm, for any fixed $1<p\le\infty$). Each time a new edge is added, it receives a random colour chosen uniformly at random and with repetition from a set of $\lceil Kn\rceil$ colours, where $K=K(d)$ is a sufficiently large fixed constant. Then, a.a.s.\ the first graph in the sequence with minimum degree at least $1$ must contain a rainbow perfect matching (for even $n$), and the first graph with minimum degree at least $2$ must contain a rainbow Hamilton cycle.
\end{abstract}

\section{Introduction}\label{s:intro}

Let $\bX=(X_1,X_2,\ldots,X_n)$ be $n$ i.i.d.\ points in $[0,1]^d$ chosen with the uniform distribution, where $d\ge2$ is fixed. 
Fix $1< p\le\infty$. Unless otherwise stated, distances and lengths in $[0,1]^d$ are measured with respect to the $\ell_p$ norm. We construct the \emph{random geometric graph} $\sG(\bX;r)$ of radius $r$ as follows. The vertices of $\sG(\bX;r)$ are indexed by $[n]$, and each pair of different vertices $i,j\in[n]$ are joined by an edge if and only if $X_i$ and $X_j$ are within ($\ell_p$-normed) distance $r$. The {\em length} of an edge $ij$ is defined to be $\|X_i-X_j\|_p$, and is always at most $r$ by construction.
With probability $1$ all points in $\bX$ are different, and fall in general position. Therefore, we will often identify vertex $i$ and point $X_i$ (i.e.~we regard $X$ as the vertex set), and assume all edges have different lengths.
Random geometric graphs (or more precisely a slight variation of the model defined above) were first introduced by Gilbert~\cite{Gil61}, and have been widely investigated ever since. They provide a theoretical model for wireless ad-hoc networks, and have relevant applications in statistics.
We refer the reader to Penrose's monograph~\cite{Pen03} and a more recent survey by Walters~\cite{Wal11} for further details and references on the subject.

We consider the natural coupling in which all $\sG(\bX;r)$ with $r\in[0,\infty)$ share one common vertex set $\bX$. We call this coupling the \emph{random geometric graph process}, and denote it by $\big(\sG(\bX;r)\big)_{r\ge0}$. Intuitively, the process starts at time $r=0$ with an empty graph on vertex set $\bX$ (almost surely, assuming that all vertices are at different positions). Then, as we increase $r$ from $0$ to $\infty$, we add edges one by one in increasing order of length. By construction, each snapshot of the process at a given time $r$ is distributed precisely as a copy of $\sG(\bX;r)$. Finally, $\sG(\bX;r)$ is deterministically the complete graph for all $r\ge D$, where $D=\|(1,1,\overset{d}{\ldots},1)\|_p  = d^{1/p}$ is the distance between two opposite corners of $[0,1]^d$.

A lot of work has been done to describe the connectivity properties of random geometric graphs in this process and the emergence of spanning subgraphs (such as perfect matchings and Hamilton cycles).
A celebrated result of Penrose~\cite{Pen99} asserts that a.a.s.\footnote{We say that a sequence of events $H_n$ holds \emph{asymptotically almost surely} (a.a.s.) if $\lim_{n\to\infty}\pr(H_n)=1$.}\ the first edge added during the process that gives minimum degree at least $k$ also makes the graph $k$-connected.
More precisely, let
\begin{align*}
\widehat r_{\delta\ge k} &= \widehat r_{\delta\ge k}(\bX) = \min \left\{ r\ge0 : \text{$\sG(\bX;r)$ has minimum degree at least $k$} \right\}
\quad\text{and}
\\
\widehat r_{\kconn} &= \widehat r_{\kconn}(\bX) = \min \left\{ r\ge0 : \text{$\sG(\bX;r)$ is $k$-connected} \right\}.
\end{align*}
Then, for every constant $k\in\nat$, 
\begin{equation}
\lim_{n\to\infty} \pr \left( \widehat r_{\delta\ge k}=\widehat r_{\kconn} \right) = 1.
\label{eq:Penrose}
\end{equation}
In view of this, Penrose (cf.~\cite{Pen03}) asked whether a.a.s.\ that first edge in the process that gives minimum degree at least $2$ (and ensures $2$-connectivity) is also responsible for the emergence of a Hamilton cycle.
A first step in this direction was achieved by D\'\i az, Mitsche and P\'erez~\cite{DMP07}, who showed (for dimension $d=2$) that, given any constant $\eps>0$, a.a.s.\ $\sG(\bX;(1+\eps)\widehat r_{\delta\ge 2})$ contains a Hamilton cycle.
Some of their ideas were recently extended by three research teams (Balogh, Bollob\'as and Walters; Krivelevich and M\" uller; and P\'erez-Gim\'enez and Wormald), who independently settled Penrose's question in the affirmative (but only two papers~\cite{BBKMW11} and~\cite{MPW11} were finally published).
In particular, a more general packing result in~\cite{MPW11} implies that
\begin{equation}
\text{a.a.s.\ }
\begin{cases}
\text{$\sG(\bX; \widehat r_{\delta\ge2})$ contains a Hamilton cycle, and}
\\
\text{$\sG(\bX; \widehat r_{\delta\ge1})$ contains a perfect matching (for even $n$).}
\end{cases}
\label{eq:HCPM}
\end{equation}
Clearly, this claim is the best possible, since any graph with minimum degree less than $1$ (less than~$2$) cannot have a perfect matching (respectively, Hamilton cycle).

In this paper we consider an edge-coloured version of the random geometric graph.
(Throughout the manuscript, we will use the term {\em edge colouring} (and other terms alike) to denote an assignment of colours to the edges, not necessarily proper in a graph-theoretical sense.)
Let $\bZ=(Z_{ij})_{1\le i<j\le n}$ be a random vector of colours, chosen independently and with replacement from a set of colours of size $c$. We use this vector $\bZ$ to colour the edges of the random geometric graph: each edge $ij$ of $\sG(\bX;r)$ ($1\le i<j\le n$) is assigned colour $Z_{ij}$. We denote this model by $\sG(\bX;\bZ;r)$. Similarly, we consider the coupled process $\big(\sG(\bX;\bZ;r)\big)_{r\ge0}$ in which, as we increase $r$ from $0$ to $\infty$, we add new edges in increasing order of length, and each new edge $ij$ is coloured according to $Z_{ij}$.

Given a graph with colours assigned to the edges, we say the graph (or its edge set) is {\em rainbow} if all edges receive different colours.
Recently, there have been many papers written on the subject of rainbow spanning structures in randomly edge coloured random graphs and digraphs (see {\em e.g.}~\cite{BBCFP, BBFP, BF13, CF02, FK15, FKMS14, FNP13, FL14, JW07}). Let $G_c(n,p)$ denote the binomial random graph $G(n,p)$ where each edge has independently been assigned a uniformly random colour from a set of size $c$. For graphs $H$ with maximum degree $\Delta = \Delta(H)$ and $n$ vertices, Ferber, Nenadov and Peter~\cite{FNP13} showed that a.a.s.\ $G_c(n,p)$ contains a rainbow copy of $H$, provided that $p = n^{1/\Delta}\polylog(n)$ and $c = (1 + o(1))e(H)$. Here, the number of colours is asymptotically optimal, whereas the bound on $p$ most likely is not. For Hamilton cycles tighter results are known.  In \cite{CF02}, Cooper and Frieze determined that a.a.s.\ $G_c(n,p)$ contains a rainbow Hamilton cycle if $p\ge 42\log n / n$ and $c\ge 21n$. This was later improved by Frieze and Loh \cite{FL14} and recently even further improved by Ferber and Krivelevich \cite{FK15} who showed that it holds when $c=(1+o(1))n$ and $p=\frac{\log n + \log\log n +\omega(1)}{n}$. Here, the number of colours is asymptotically optimal and the bound on $p$ is optimal. Bal and Frieze \cite{BF13} examined the case when the number of colours is exactly optimal, showing that $G_n(n,p)$ a.a.s.\ contains a rainbow Hamilton cycle as long as $p =\Omega\of{\frac{\log n}{n}}$.

In this manuscript, we investigate the emergence of rainbow spanning structures in the random geometric graph.
Our main contribution is to extend~\eqref{eq:HCPM} to a rainbow context. We show that a.a.s.\ the first edge in the edge-coloured random geometric graph process $\big(\sG(\bX;\bZ;r)\big)_{r\ge0}$ that gives minimum degree at least $2$ also creates a
rainbow Hamilton cycle, provided that the number of colours is at least $c=\lceil Kn\rceil$, where $K=K(d)>0$ is a sufficiently large constant.
Similarly, under the same assumptions (with $n$ even), the first edge in the process that ensures that the minimum degree is at least $1$ creates a rainbow perfect matching.
To state the result more precisely, let 
\begin{align*}
&
\widehat r_{\RPM} = \widehat r_{\RPM}(\bX,\bZ) =  \inf \left\{ r\ge0 : \text{$\sG(\bX;\bZ;r)$ contains a rainbow perfect matching} \right\}
\quad\text{and}
\\
&\widehat r_{\RHC} = \widehat r_{\RHC}(\bX,\bZ) =  \inf \left\{ r\ge0 : \text{$\sG(\bX;\bZ;r)$ contains a rainbow Hamilton cycle} \right\},
\end{align*}
where we use the convention that $\inf\, \emptyset = \infty$. Note that whenever $\widehat r_{\RPM}<\infty$ ($\widehat r_{\RHC}<\infty$) the infimum in the above definition is actually a minimum, and it is precisely the length of the first edge in the process $\big(\sG(\bX;\bZ;r)\big)_{r\ge0}$ that creates a rainbow perfect matching (respectively, rainbow Hamilton cycle).
\begin{thm}\label{thm:main}
Given a fixed integer $d\ge2$, there exists a sufficiently large constant $K=K(d)>0$ satisfying the following.
Let $\bX=(X_1,X_2,\ldots,X_n)$ be $n$ i.i.d.\ points in $[0,1]^d$ chosen uniformly at random, and let $\bZ=(Z_{ij})_{1\le i<j\le n}$ be a random vector of colours, choosen independently and with replacement from a set of colours of size $\lceil Kn\rceil$. Consider the random geometric graph process $\big(\sG(\bX;r)\big)_{r\ge0}$ (for any fixed $\ell_p$-normed distance, $1< p\le\infty$) with a random colouring of the edges given by $\bZ$. Then,
\[
\lim_{n\to\infty} \pr \big( \widehat r_{\RHC}(\bX,\bZ) = \widehat r_{\delta\ge2}(\bX)  \big) = 1,
\]
and for even $n$
\[
\lim_{\substack{n\to\infty\\\text{($n$ even)}}} \pr \big( \widehat r_{\RPM}(\bX,\bZ) = \widehat r_{\delta\ge1}(\bX)  \big) = 1.
\]
\end{thm}
\begin{remark}
For $p=1$, we can only claim that
\begin{equation}\label{eq:p1}
\lim_{n\to\infty} \pr \big( \widehat r_{\RHC}(\bX,\bZ) = \widehat r_{\2conn}(\bX)  \big) = 1
\quad\text{and}\quad
\lim_{\substack{n\to\infty\\\text{($n$ even)}}} \pr \big( \widehat r_{\RPM}(\bX,\bZ) = \widehat r_{\1conn}(\bX)  \big) = 1,
\end{equation}
and in fact it is not known whether~\eqref{eq:Penrose} holds (see~\cite{Pen99} and~\cite{Pen03}). We include a justification of~\eqref{eq:p1} in Section~\ref{sec:p1} for completeness.
\end{remark}
Combining Theorem~\ref{thm:main} and Theorem~8.4 in~\cite{Pen03}, we immediately obtain the limiting probabilities of having a rainbow perfect matching and having a rainbow Hamilton cycle, assuming that the number of colours is sufficiently large.
\begin{cor}\label{cor:limdisPM}
Under the same assumptions of Theorem~\ref{thm:main}, put
\[
r=\sqrt[d]{\frac{(2/d)\log n + (3-d-2/d) \log\log n + x}{2^{2-d}\theta n}},
\]
and let $f = \log \left( 2^{1-2/d}(\theta d)^{3-2/d}{\theta'}^{d-2} \big/ \binom{d}{2} \right)$,
where $\theta$ and $\theta'$ are the volumes of the $d$-dimensional and $(d-1)$-dimensional unit $\ell_p$-balls, respectively. Then 
\[
\lim_{\substack{n\to\infty\\\text{($n$ even)}}} \pr\big(\text{$\sG(\bX;\bZ;r)$ has a rainbow perfect matching}\big) =
\begin{cases}
0 & x\to -\infty;
\\
\exp\left(- e^{-\alpha-f}\right) & x\to \alpha;
\\
1 & x\to\infty.
\end{cases}
\]
\end{cor}
\begin{cor}\label{cor:limdisHC}
Under the same assumptions of Theorem~\ref{thm:main}, put
\[
r=\sqrt[d]{\frac{(2/d)\log n + (4-d-2/d) \log\log n + y}{2^{2-d}\theta n}},
\]
and let $f$, $\theta$ and $\theta'$ be as in Corollary~\ref{cor:limdisPM}. Then if $d\ge3$,
\[
\lim_{n\to\infty} \pr\big(\text{$\sG(\bX;\bZ;r)$ has a rainbow Hamilton cycle}\big) =
\begin{cases}
0 & y\to -\infty;
\\
\exp\left(- 2e^{-\alpha-f}/d\right) & y\to \alpha;
\\
1 & y\to\infty.
\end{cases}
\]
Otherwise if $d=2$,
\[
\lim_{n\to\infty} \pr\bigg(\parbox{10.1em}{$\sG(\bX;\bZ;r)$ has a\\ rainbow Hamilton cycle}\bigg) =
\begin{cases}
0 & y\to -\infty;
\\
\exp\left(- e^{-\alpha/2}\left(e^{-\alpha/2}+\frac{2\sqrt\theta}{\theta'}\right)\right) & y\to \alpha;
\\
1 & y\to\infty.
\end{cases}
\]
\end{cor}
In Section~\ref{sec:tess}, we will discuss how we partition $[0,1]^d$ into a grid of small $d$-dimensional cubic cells.   Having this partition will simplify our task by allowing us to locally search for short rainbow paths or cycles within each cell or small cluster of cells. In Section~\ref{sec:localtasks}, we will show how to find these paths or cycles. In Section~\ref{sec:connect}, we will show how to connect the pieces together into one rainbow Hamilton cycle. A simple adaptation of the argument can be used to build a rainbow perfect matching. Finally, in Section~\ref{sec:p1} we will discuss the case $p=1$, and pose some open questions in Section~\ref{sec:future}.

%
%
\section{Tessellation and graph of cells}\label{sec:tess}
The main goal in this section is to prove Lemmas~\ref{lem:paths} and~\ref{lem:paths2}, which will be crucial in the construction of the rainbow Hamilton cycle and perfect matching. We will adapt and extend many ideas from~\cite{MPW11}.
Throughout the paper, $d\ge2$ and the $\ell_p$-norm ($1<p\le\infty$) in $\real^d$ are fixed.
Note that the volume $\theta$ of the unit $d$-dimensional $\ell_p$-ball satisfies
\begin{equation}\label{eq:thetabound}
2^d/d! \le \theta \le 2^d
\end{equation}
since the $\ell_p$-ball contains the $\ell_1$-ball and is contained in the $\ell_\infty$-ball, which have volume $2^d/d!$ and $2^d$ respectively. We will also  make frequent use of the following inequality throughout the argument, often without explicitly mentioning it. For any $X\in\real^d$,
\begin{equation}\label{eq:lplinf}
\|X\|_\infty \le \|X\|_p \le d \|X\|_\infty.
\end{equation}
Moreover, we will pick a sufficiently small constant $\eps=\eps(d)>0$ so that several requirements in the argument are met. Later we will choose $K$ sufficiently large with respect to this $\eps$ (recall $\lceil Kn \rceil$ is the number of colours). We remark that our choice of $\eps$ does not depend on $p$ or other parameters that may be introduced later in the statements.

We use the standard $o()$, $O()$, $\Theta()$ and $\Omega()$  asymptotic notation as $n\to\infty$ with the following extra considerations. We do not assume any sign on a sequence $a_n$ satisfying $a_n=o(1)$ or  $a_n=O(1)$, but on the other hand a sequence satisfying $a_n=\Theta(1)$ or $a_n=\Omega(1)$ is assumed to be positive for all but finitely many $n$. Furthermore, the constants involved in the bounds of the definitions of $O()$, $\Theta()$ and $\Omega()$ may depend on $d$, but not on $p$ or $\eps$. Whenever these constants depend on our choice of $\eps$ (in addition to $d$), we use the alternative notation $O_\eps()$, $\Theta_\eps()$ and $\Omega_\eps()$ instead. Our asymptotic statements are always uniform for all $1<p\le\infty$ as a consequence of bounds~\eqref{eq:thetabound} and~\eqref{eq:lplinf}.

Let $\omega = \omega(n) \to\infty$ be some function tending to infinity sufficiently slowly (in particular, $\omega = o(\log \log n)$). We define $r_0$ and $r_1$ by
\begin{align*}
\theta n {r_0}^d &= (2^{d-1}/d)\log n + 2^{d-2}(3-d-2/d) \log\log n - \omega
\qquad\text{and}\qquad
\\
\theta n {r_1}^d &= (2^{d-1}/d)\log n + 2^{d-2}(4-d-2/d) \log\log n + \omega.
\end{align*}
Then, by Theorem~8.4 in~\cite{Pen03}, the respective lengths $\widehat r_{\delta\ge1}$ and $\widehat r_{\delta\ge2}$ of the critical edges of the process  $\big(\sG(\bX;r)\big)_{r\ge0}$ that give minimum degree 1 and 2 satisfy
\begin{equation}\label{eq:r0rd2}
r_0 \le\widehat r_{\delta\ge1} \le\widehat r_{\delta\ge2}\le r_1 \sim r_0
\qquad
\text{a.a.s.}
\end{equation}
Our argument will use edges of length at most $r_0$ to construct most of the rainbow perfect matching or Hamilton cycle, and only use a few longer edges of length up to $\widehat r_{\delta\ge1}$ or $\widehat r_{\delta\ge2}$ at some exceptional places.

Let $s'=(2\eps d\theta)^{1/d} r_0 /2$.
We tessellate $[0,1]^d$ into $d$-dimensional cubic \emph{cells} of side length
\[
s=\left\lceil  ( s' )^{-1} \right\rceil^{-1} \sim s' = \Theta(\eps^{1/d} r_0),
\]
that is, of volume
\[
s^d \sim \eps d 2^{1-d} \theta {r_0}^d  \sim \eps \log n/ n,
\]
arranged in a grid fashion. Let $\cC$ denote the set of cells. There are $\Theta_\eps(n/\log n) = o(n)$ cells in $\cC$, where we recall that the constant hidden in the $\Theta_\eps(\cdot)$ notation depends on $\eps$ (and $d$).
Clearly (assuming that $\eps$ is sufficiently small given $d$ and by~\eqref{eq:thetabound}), the vertices inside each cell induce a clique in $\sG(\bX;r_0)$, and in fact a stronger property holds in view of the following definition.
\begin{defn}
The  {\em graph of cells} $\sG_\cC$ is a graph with vertex set $\cC$ (i.e.~the set of cells of the tessellation), and two cells are adjacent in $\sG_\cC$ if they are at ($\ell_p$-normed) distance at most $r_0 - 2ds = (1- \Theta(\eps^{1/d})) r_0$.
\end{defn}
This implies that for any pair of adjacent cells and any pair of points $X_i,X_j\in \bX$ that belong to these cells, $X_i$ and $X_j$ must be adjacent in the graph $\sG(\bX;r_0)$. (This is true regardless of the $\ell_p$-norm being used, in view of~\eqref{eq:lplinf}.)
The degree of a cell $C$ in the graph of cells $\sG_\cC$ is at most the number of cells contained in a ball of radius $r_0$ centered at the center of $C$.  As each cell has volume $\Theta(\eps {r_0}^d)$ and the ball of radius $r_0$ has volume $\theta {r_0}^d$, we deduce that
\begin{equation}\label{eq:Delta}
\text{the maximum degree of the graph of cells is $\Delta(\sG_\cC)  = O(1/\eps)$.}
\end{equation}
Note that the number of points of $\bX$ that fall into each cell is distributed as $\text{Bin}(n, s^d)$ with expectation $s^dn\sim \eps \log n$. Then, we can easily bound the maximum number of points in a cell.
\begin{lem}\label{lem:maxdensity}
A.a.s.\ no cell in $\cC$ contains more than $\log n$ vertices of $\bX$.
\end{lem}
\begin{proof}
\[
\pr\left[\text{Bin}(n, s^d) \ge \log n \right] \le \binom{n}{\lceil\log n\rceil} s^{d\lceil\log n\rceil} \le
\big( e\eps + o(1)\big)^{\lceil\log n\rceil} = o(1/n),
\]
since $e\eps < e^{-1}$ (assuming $\eps$ is sufficiently small). Thus, a union bound over all $O_\eps(n/\log n)$ cells shows that a.a.s.\ none has more than $\log n$ many points.
\end{proof}
Similarly, we can perform analogous calculations to bound the number of vertices of $\bX$ that fall inside of a ball of radius $\ell r_1$ centered around a vertex $X\in\bX$, take a union bound over all $n$ choices of $X$, and conclude the following.
\begin{lem}\label{lem:DeltaRGG}
Given any constant $\ell \in\nat$, the maximum degree of the power graph ${\sG(\bX;r_1)}^\ell$ is a.a.s.\ $O(\log n)$.
\end{lem}

\begin{defn}\label{def:densesparse}
We say a cell is \emph{dense} if it has at least $\eps^3\log n$ points of $\bX$ in it, and is otherwise \emph{sparse}.
\end{defn}
Note that this definition is different from the corresponding notions in~\cite{MPW11} and~\cite{BBKMW11}, which only require dense cells to contain a large but constant number of points of $\bX$. We will show we cannot have too many sparse cells or too large connected sets of sparse cells in the graph of cells $\sG_\cC$ (or in its $\ell$-th power ${\sG_\cC}^\ell$, for a fixed $\ell\in\nat$).
We shall also take into account whether these cells are ``close'' to the boundary of the cube $[0,1]^d$.
To make this precise, define $F^\beta_j = [0,1]^{j-1}\times\{\beta\}\times[0,1]^{d-j}$ for $\beta\in\{0,1\}$ and $j\in\{1,2,\ldots,d\}$. These are the $2d$ facets (i.e.~$(d-1)$-dimensional faces) of the boundary of $[0,1]^d$.
\begin{lem}\label{lem:fewsparse}
\hspace{0cm}
\begin{enumerate}
\item A.a.s.\ the number of sparse cells is at most $n^{1-\eps /2}$.
\item Moreover, for any arbitrary constants $A>0$ and $\ell\in\nat$, a.a.s.\ the power graph ${\sG_\cC}^{\ell}$
\begin{enumerate}
\item
has no connected set $\cS$ of at least $(1+\eps)/\eps$ cells which are all sparse;
\item has no connected set $\cS$ of at least $\frac{d-i}{d} (1+\eps)/\eps$ cells which are all sparse and such that some cell in $\cS$ lies within distance $A r_0$ from at least $i$ facets of $[0,1]^d$ (for $ i\in\{0,1,\ldots,d-1\}$); 
\item has no sparse cell within distance $A r_0$ from $d$ facets of $[0,1]^d$. 
\end{enumerate}
\end{enumerate}
\end{lem}
\begin{remark}\label{rem:fewsparse}\hspace{0cm}
\begin{enumerate}
\item
The property ``of at least $\frac{d-i}{d} (1+\eps)/\eps$ cells'' in the statement can be replaced by ``of total volume at least $(1+\eps)\frac{d-i}{2^{d-1}}\theta{r_0}^d$'', and the claim is still valid.
\item
Note that this lemma provides an analogue of Lemma~4 in~\cite{MPW11}. However, the latter gives a $\frac{d-i}{d} (1+\alpha)/\eps$ bound on the size of $\cS$, for $\alpha$ arbitrarily small (possibly much smaller than any fixed function of $\eps$). Here, we cannot achieve that, given our more restrictive definition of dense cell (i.e.~sparse cells are more abundant). However, the current statement will suffice for our purposes.
In particular, Lemma~4 in~\cite{MPW11} is used in the proof of Lemma~5 in~\cite{MPW11} with $\alpha=\Theta(\eps^{1/d})$, which is greater than $\eps$ (if $\eps$ is sufficiently small), and thus this situation is covered by our present statement.
Hence, Lemma~5 in~\cite{MPW11} is still valid with our definition of dense cells, since we can replace all uses of Lemma~4~\cite{MPW11} in the proof by its counterpart in this manuscript.

\end{enumerate}
\end{remark}
\begin{proof}[Proof of Lemma~\ref{lem:fewsparse}]
Recalling that the number of points in any fixed cell is distributed as $\text{Bin}(n, s^d)$, the probability that a cell is sparse is  
(with the convention that $(a/0)^0=1$ for all $a\in\real$)
\begin{align}
\sum_{k=0}^{\lceil\eps^3 \log n\rceil-1} \binom{n}{k} s^{dk} \of{1- s^d}^{n-k} 
&\le \sum_{k=0}^{\lceil\eps^3 \log n\rceil-1} \bfrac{ens^d}{k(1-s^d)}^k  e^{- s^d n}
\notag\\
&\le \lceil\eps^3 \log n\rceil  \bfrac{e+o(1)}{\eps^2}^{\eps^3\log n}  e^{- (\eps+o(1))\log n}
\notag\\
&=  n^{- \eps + \eps^3\log(e/\eps^2) +o(1)}
\notag\\
&\le  n^{- \eps(1-\eps/2)} \quad \text{(for large n),}
\notag
\end{align}
provided that $\eps$ is chosen so that $\eps\log(e/\eps^2) < 1/2$. Note that this is possible since $\eps\log(e/\eps^2) \to 0$ as $\eps\to 0$.  In particular, this condition implies that $\eps < 1$ and so the number of sparse cells is a.a.s.\ at most $n^{1-\eps/2}$ by Markov's inequality. This proves part~1.

Fix $A>0$, $\ell\in\nat$ and $ i\in\{0,1,\ldots,d-1\}$. In order to prove~2(b), it is enough to show that ${\sG_\cC}^{\ell}$ contains no connected set of exactly $\lceil \frac{d-i}{d} (1+\eps)/\eps\rceil$ sparse cells within distance $A r_0$ from $i$ facets of $[0,1]^d$.
Observe that the events that two or more cells are sparse are negatively correlated. Therefore,
\begin{equation}
\label{eq:psparse}
\text{the probability that $k$ given cells are sparse is at most $n^{-k\eps(1-\eps/2)}$.}
\end{equation}
If we also choose $\eps$ small enough so that $(1-\eps/2)(1+\eps) > 1$, the probability that a given set of $\lceil((d-i)/d)(1+\eps)/\eps\rceil$ cells are all sparse is $o(n^{-(d-i)/d})$. Since there are only $O_\eps\of{(n/\log n)^{(d-i)/d}}$ possible connected sets of $\lceil((d-i)/d)(1+\eps)/\eps\rceil$ cells in the power graph ${\sG_\cC}^{\ell}$ lying within distance $A r_0$ from $i$ facets of $[0,1]^d$, we can take the union bound and complete the proof of part~2(b). Part~2(a) follows as a particular case of part~2(b) taking $i=0$. 
Finally, the expected number of sparse cells within distance $A r_0$ from $d$ facets of $[0,1]^d$ is at most $O_\eps(1)$ times $n^{- \eps(1-\eps/2)}$, which is $o(1)$. This immediately yields part~2(c) and completes the proof.
\end{proof}

Given a set of cells $\cS\subseteq\cC$, we denote by $\sG_\cC[\cS]$ the subgraph of the graph of cells induced by $\cS$.
\begin{defn}\label{defn3}
Let $\cD$ be the set of dense cells, and let $\cG$ be the set of cells in the largest component of $\sG_\cC[\cD]$ (i.e.~the subgraph of $\sG_\cC$ induced by dense cells). If there are two or more such largest components, pick one according to any arbitrary deterministic rule (we will see that a.a.s.\ $\cG$ is very large, so the choice of $\cG$ is unique). Call cells in $\cG$ {\em good}. Cells that are not good, but are adjacent in the graph of cells to some good cell are called {\em bad}. Bad cells must be sparse by construction.
Cells that are not adjacent to good cells are called {\em ugly}. Note that ugly cells may be dense or sparse. Let $\cB$ and $\cU$ denote the set of bad and ugly cells, respectively.
\end{defn}

As a crucial ingredient in our argument, we will use Lemma~5 in~\cite{MPW11}, which shows that a.a.s.\  ugly cells (which are called ``bad" in that paper) appear in small clusters far enough from each other. (Note that this result is stated for dimension $d=2$, and then extended to general $d\ge2$ in Section~4 of~\cite{MPW11}.) Unfortunately, the definition of dense cell we use in the present manuscript is more restrictive than the one in~\cite{MPW11} (they only require a dense cell to contain at least $M$ points, for a large constant $M>0$; whilst here we require at least $\eps^3\log n$ points).
In order to overcome this minor obstacle, we simply observe that our Lemma~\ref{lem:fewsparse} extends Lemma~4 in~\cite{MPW11} to a less restrictive notion of sparse cell (although with a slightly weaker bound).
In view of Remark~\ref{rem:fewsparse}(2), the proof of Lemma~5 in~\cite{MPW11} is also valid in our setting by trivially replacing Lemma~4 in~\cite{MPW11} by Lemma~\ref{lem:fewsparse} of the present paper. Hence, adapting Lemma~5 in~\cite{MPW11} to our current notation, we obtain the following statement.
\begin{lem}[\cite{MPW11}]\label{lem:ugly1}
A.a.s.\ all connected components of $\sG_\cC[\cU]$ have $\ell_\infty$-diameter at most $4d^2 s$. 
\end{lem}
Next, we obtain useful bounds on the number of bad and ugly cells.
\begin{lem}\label{lem:badugly}
A.a.s.\ there are at most $n^{1-\eps/2}$ bad cells and at most $n^{O\left(\eps^{1/d}\right)}$ ugly cells.
\end{lem}
In particular, this implies that a.a.s.\ our choice of $\cG$ in Definition~\ref{defn3} was unique.
\begin{proof}
Since bad cells are sparse by definition, the first part of the statement follows trivially from Lemma~\ref{lem:fewsparse}(1).

To prove the second part, consider a cell $C=[a_1,a_1+s]\times\cdots\times[a_d,a_d+s]$ which is at distance at most $r_0$ from exactly $i$ facets of the cube $[0,1]^d$, for some $0\le i\le d$. Without loss of generality, assume these facets are precisely $F^0_1,\ldots,F^0_i$. Let $P=(a_1+s,\ldots,a_d+s)$ (i.e.~$P$ is the point in $C$ with largest coordinates), and let $B=B(0;r_0-4ds)$ denote the $\ell_p$ ball with centre $0\in\real^d$ and radius $r_0-4ds$.
By construction, the set
\[
S = P + \left(  B \cap \left( [0,\infty)^i\times\real^{d-i} \right) \right)
\]
contains precisely those points within distance $r_0-4ds$ of $P$ and ``above'' $P$ with respect to the first $i$ coordinates. Moreover, $S$ is fully contained in the cube $[0,1]^d$, and every cell $C'\ne C$ intersecting $S$ must belong to the set $\cN$ of cells that are adjacent to $C$ in the graph of cells $\sG_\cC$.
Therefore, the number of cells in $\cN$ satisfies
\begin{align}
|\cN| + 1 &\ge \vol(S)/s^d
\notag\\
&= 2^{-i}\theta {r_0}^d(1-4ds/r_0)^d/s^d
\notag\\
&\ge  \left(2^{d-i-1}/d +o(1)\right) \eps^{-1} (1-4d^2s/r_0)
\notag\\
&\ge \left(2^{d-i-1}/d\right) \eps^{-1} (1 - \alpha\eps^{1/d}) \quad \text{(eventually)}
\label{eq:calN}
\end{align}
for some constant $\alpha=\alpha(d)>0$, where we used ${r_0}/s \sim 2( \eps d 2 \theta )^{-1/d}$ and the fact that $(1-x)^d\ge1-dx$ for $0\le x\le1$.
Furthermore, let $E_C$ be the event that $C$ is ugly and is adjacent in $\sG_\cC$ to at most $(4d^2)^d-1$ other ugly cells. This event implies that there must be a set $\cN'$ of at least
\begin{equation}
|\cN| - (4d^2)^d \ge
\left(2^{d-i-1}/d\right) \eps^{-1} (1 - 2\alpha\eps^{1/d})
\label{eq:Nprime}
\end{equation}
sparse cells adjacent to $C$. Note that the last line follows since we can choose $\epsilon$ sufficiently small, given $d$ (and $\alpha$).  
Thus, by~\eqref{eq:psparse} and summing over all $O_\eps(1)$ possible choices of such $\cN'$, we obtain that the probability of $E_C$ is at most
\[
O_\eps(1) \cdot n^{- \left(2^{d-i-1}/d\right) (1 - 2\alpha\eps^{1/d}) (1-\eps/2)}
\le n^{- \left(2^{d-i-1}/d\right) \left(1 - \Theta\left(\eps^{1/d}\right)\right) } \quad \text{(eventually)},
\]
assuming that $\eps$ is sufficiently small. Hence, summing over $i$ and over the $O_\eps \left( (n/\log n)^{(d-i)/d}\right)$ possible choices of $C$, we show that the expected number of cells that are ugly and are adjacent in $\sG_\cC$ to at most $(4d^2)^d$ other ugly cells is
\[
\sum_{i=0}^d O_\eps \left( (n/\log n)^{(d-i)/d}  n^{- \left(2^{d-i-1}/d\right) \left(1 - \Theta\left(\eps^{1/d}\right)\right) } \right)
=  n^{ O\left(\eps^{1/d}\right) },
\]
since $d-i \le 2^{d-i-1}$.
By Markov's inequality and in view of Lemma~\ref{lem:ugly1} (which implies that a.a.s.\ there are no
ugly cells that are adjacent to more than $(4d^2)^d$ other ugly cells), we conclude that a.a.s.\ there are at most $n^{ O\left(\eps^{1/d}\right) }$ ugly cells. This finishes the proof of the second statement.
\end{proof}

\begin{lem}\label{lem:ugly2}
Given any constant $A>0$, a.a.s.\ every two ugly cells lying in different components of $\sG_\cC[\cU]$ are at ($\ell_p$-normed) distance at least $Ar_0$ apart.
\end{lem}

\begin{proof}
We will assume that the a.a.s.\ conclusions of Lemmas~\ref{lem:fewsparse} and~\ref{lem:ugly1} hold, and deterministically prove that any pair of nonadjacent cells in $\sG_\cC$ at distance less than $Ar_0$ cannot both be ugly.
Thus, consider any two different cells $C=[a_1,a_1+s]\times\cdots\times[a_d,a_d+s]$ and $C'=[a'_1,a'_1+s]\times\cdots\times[a'_d,a'_d+s]$ that are not adjacent in the graph of cells $\sG_\cC$, but are at distance less than $Ar_0$ from each other. Suppose that there are exactly $i$ facets of the $d$-cube $[0,1]^d$ (for some $0\le i\le d$) at distance less than $r_0$ from  $C$ or $C'$. Assume, without loss of generality, that these facets are precisely $F^0_1,\ldots,F^0_i$. In particular, both $C$ and $C'$ must be at distance at most $(A+1)r_0+2ds \le (A+2)r_0$ from these $i$ facets (for $\eps$ sufficiently small given $d$), and at distance at least $r_0$ from any other facet.

We will proceed in a similar fashion as in the proof of Lemma~\ref{lem:badugly} in order to describe a large set of cells that are adjacent to $C$ or $C'$.
Let $P=(a_1+s,\ldots,a_d+s)$ and $P'=(a'_1+s,\ldots,a'_d+s)$ be respectively the points in $C$ and $C'$ with largest coordinates. 
The hyperplane in $\real^d$ orthogonal to vector $P'-P$ (with respect to the Euclidean inner product) and passing through the origin splits $\real^d$ into two halfspaces
\[
H = \{Q\in\real^d : \langle Q, P'-P\rangle \le 0 \}
\qquad\text{and}\qquad
H' = \{Q\in\real^d : \langle Q, P'-P\rangle \ge 0 \}.
\]
Consider the set
\[
B_i =   B \cap \left( [0,\infty)^i\times\real^{d-i} \right)
\]
where $B=B(0;r_0-4ds)$ denotes the $\ell_p$ ball with centre $0\in\real^d$ and radius $r_0-4ds$. Since $B_i$ has volume 
$2^{-i}\theta{r_0}^d(1-4ds/r_0)^d$, at least one of the sets $B_i\cap H$ or $B_i\cap H'$ has volume at least half of this. Assume it is $B_i\cap H'$ (by  otherwise reversing the roles of $P$ and $P'$).
Define the sets
\[
S = P + B_i
\qquad\text{and}\qquad
S' = P' + B_i\cap H'.
\]
Since $C$ and $C'$ are not adjacent in $\sG_\cC$, $P$ and $P'$ must be at distance greater than $r_0-2ds$. In particular, $P'\notin S$. Moreover, by our choice of $H'$, any other point in $S'$ is further away from $P$ than $P'$ is, so $S$ and $S'$ must be disjoint sets. Let $\cN$ be the set of cells different than $C$ and $C'$ intersecting $S\cup S'$. By construction, every cell in $\cN$ must be adjacent in the graph of cells $\sG_\cC$ to $C$ or $C'$. Moreover, since $\vol(S \cup S') \ge (3/2)\vol(S)$, we can use~\eqref{eq:calN} to infer
\[
|\cN| + 2 \ge \vol(S \cup S')/s^d
\ge (3/2)\left(2^{d-i-1}/d\right) \eps^{-1} (1 - \alpha \eps^{1/d})
\ge (4/3)\frac{d-i}{d} \eps^{-1},
\]
for $\eps$ sufficiently small given $d$, where we used that $d-i \le 2^{d-i-1}$.
In view of Lemma~\ref{lem:ugly1} (assuming its a.a.s.\ conclusion is true), if both $C$ and $C'$ are ugly, all but at most $2(4d^2)^d$ cells in $\cN$ must be sparse. 
Then, we obtained a set of at least $(4/3)\frac{d-i}{d} \eps^{-1} - 2(4d^2)^d \ge (5/4)\frac{d-i}{d} \eps^{-1}$ sparse cells within distance at most $(A+2)r_0$ from each other (i.e.~pairwise adjacent in the power graph ${\sG_\cC}^\ell$ for $\ell=\lceil A+2\rceil$).
This contradicts the a.a.s.~conclusion of Lemma~\ref{lem:fewsparse}. 
\end{proof}

\begin{lem}\label{lem:goodaroundugly}
A.a.s.\ the following holds. Given any ugly cell $C$, let $\cG'\subseteq\cG$ be the set of good cells at $\ell_\infty$-distance at most $3r_0$ from $C$. Then, the subgraph of $\sG_\cC$ induced by $\cG'$ is connected and has diameter (as a graph) at most $2(20d)^d$.
\end{lem}
\begin{proof}
Let $t'=r_0/(3d)$. Tesselate $[0,1]^d$ by $d$-dimensional cubes of side $t = \lceil(t')^{-1}\rceil^{-1}\sim r_0/(3d)$. We call these cubes {\em boxes} to distinguish them from the cells.
Note that a cell may intersect a box and not be fully contained in it. Regardless of that, each box contains at least
\begin{equation}
(t/s-2)^d \ge \eps^{-1}/ \big(\theta (2d)^{d+1}\big)
\label{eq:cellsinbox}
\end{equation}
cells. 
We say that two different boxes $B$ and $B'$ are adjacent if the set $B\cup B'$ is topologically connected (i.e.~they share at least one point). (Recall that the term adjacent has a different meaning for cells.)
If two different cells $C$ and $C'$  intersect the same box or two adjacent boxes, then $C$ and $C'$ must be at $\ell_p$-distance at most $2dt\sim 2r_0/3$, and thus must be adjacent in the graph of cells $\sG_\cC$.

Fix an ugly cell $C$. We will assume that all the a.a.s.\ properties leading to the conclusions of Lemmas~\ref{lem:fewsparse}, \ref{lem:ugly1}, \ref{lem:badugly} and~\ref{lem:ugly2} hold, and deterministically prove the statement for cell $C$.
Let $R_1\subseteq[0,1]^d$ be the union of all boxes at $\ell_\infty$-distance between $1.1r_0$ and $3r_0-t$ from cell $C$. Since the side of a box is $t< r_0/5$, then $R_1$ is topologically connected. Moreover, $[0,1]^d\setminus R_1$ has two connected components $R_0$ and $R_2$ with, say, $C\subseteq R_0$.
(It is worth noting that these properties of $R_1$ hold regardless of how close cell $C$ is from some facets of $[0,1]^d$, but this may cease to be true if we replaced $\ell_\infty$ by some other $\ell_p$ in the definition of $R_1$.)
Let $\cG_0$ and $\cG_2$ be the set of good cells contained in $R_0$ and $R_2$ (respectively), and let $\cG_1$ be the set of good cells that intersect $R_1$.
Every good cell in $\cG$ must belong to exactly one of the sets $\cG_0$, $\cG_1$ and $\cG_2$.
Observe that every cell  in $\cG_0$ is at $\ell_\infty$-distance at most $1.1r_0+t \le1.3r_0$ from $C$, and every cell  in $\cG_2$ is at $\ell_\infty$-distance at least $3r_0-t\ge 2.8r_0$ from $C$.
Therefore, if $C_0\in\cG_0$ and $C_2\in\cG_2$, then $C_0$ and $C_2$ must be at $\ell_\infty$-distance at least $1.5r_0-ds\ge r_0$. In particular, no cell in $\cG_0$ is adjacent to any cell in $\cG_2$ with respect to the graph of cells $\sG_\cC$.

\begin{claim}\label{claim:R1}
Every box $B\subseteq R_1$ contains a good cell.
\end{claim}

Since $R_1$ is a topologically connected union of boxes and pairs of cells contained in adjacent boxes are adjacent in the graph of cells, Claim~\ref{claim:R1} implies that $\cG_1$ must induce a connected subgraph of  $\sG_\cC$.
Recall that the graph induced by the set $\cG=\cG_0\cup\cG_1\cup\cG_2$  is also connected by the definition of good cells. Hence, since there are no edges between $\cG_0$ and $\cG_2$ in the graph of cells, we deduce that $\cG_0\cup\cG_1$ also induces a connected graph.
Let $\cG'$ be the set of good cells at $\ell_\infty$-distance at most $3r_0$ from cell $C$.
Observe that $\cG_0\cup\cG_1 \subseteq \cG'$. Moreover, and every cell $C'\in\cG' \setminus (\cG_0\cup\cG_1)$ must intersect a box $B'$ which is adjacent to a box $B\subseteq R_1$, so $C'$ is adjacent in the graph of cells to some cell in $\cG_0\cup\cG_1$.
Hence, $\cG'$ induces a connected graph as well.

We proceed to bound the diameter (as a graph) of $\sG_\cC[\cG']$. Consider any two cells $C',C''\in\cG'$ and a path of cells $C'=C_1,C_2,\ldots,C_j=C''$ in $\cG'$ of minimal length. If two cells in the path intersect the same box, then they are adjacent in the graph of cells, so they must be consecutive in the path by our minimal length assumption. Similarly, we cannot have more than two consecutive cells in the path intersecting one same box. Therefore, we deduce that the number $j$ of cells in the path is at most twice the number of boxes that may be potentially intersected by cells in $\cG'$, which is at most
\[
2((6r_0+s)/t+2)^d \le 2(20d)^d.
\]
This gives the desired upper bound on the diameter of the graph.

It only remains to prove Claim~\ref{claim:R1}. In order to do so, suppose our ugly cell $C$ is at distance at most $r_0$ from exactly $i$ facets of the cube $[0,1]^d$, for some $0\le i\le d$. By~\eqref{eq:Nprime} in the proof of Lemma~\ref{lem:badugly}, we can find a set $\cN'$ of at least $\left(2^{d-i-1}/d\right) \eps^{-1} (1 - 2\alpha\eps^{1/d})$ sparse cells within $\ell_p$-distance at most $r_0$ from $C$, where $\alpha=\alpha(d)$ is a positive constant.

Pick any box $B\subseteq R_1$. Every cell contained in $B$ is at $\ell_\infty$-distance (and thus at $\ell_p$-distance) at least $1.1r_0$ from $C$, so it cannot belong to $\cN'$. Suppose that all cells contained in $B$ are sparse. Then, by~\eqref{eq:cellsinbox}, we have at least
\[
\left(2^{d-i-1}/d\right) \eps^{-1} (1 - 2\alpha\eps^{1/d}) + \eps^{-1}/ (\theta (2d)^{d+1}) \ge \left(2^{d-i-1}/d\right) \eps^{-1} (1 +\eps) \ge \frac{d-i}{d} (1 +\eps) \eps^{-1}
\]
sparse cells at $\ell_\infty$-distance at most $3r_0$ from $C$. Thus, these sparse cells are within $\ell_p$-distance $d(6r_0+s)\le7d(r_0-2ds)$ from each other and at distance at most $d(3r_0+s)+r_0 \le (4d+1)r_0$ from $i$ facets of $[0,1]^d$. This contradicts the a.a.s.\ conclusion of Lemma~\ref{lem:fewsparse} (with parameters $A=(4d+1)$ and $\ell=7d$). Therefore, every box $B$ in $\cR$ must contain at least one dense cell, which must be also good in view of Lemmas~\ref{lem:ugly1} and~\ref{lem:ugly2}. This proves Claim~\ref{claim:R1}, and completes the proof of the lemma.
\end{proof}

We say that a collection $\fP$ of paths in a graph covers a vertex if the vertex belongs to some path in $\fP$.
The next lemma provides us with an appropriate collection of paths $\fP$ that covers the ugly vertices and will be a crucial ingredient in the construction of a rainbow Hamilton cycle.
\begin{lem}\label{lem:paths}
Let $\bX_\cU$ denote the set of vertices in $\bX$ that belong to ugly cells. Given any constant $A>0$, a.a.s.\ there is a collection $\fP$ of vertex-disjoint paths in $\sG(\bX; \widehat r_{\delta\ge2})$ such that:
\begin{enumerate}
\item
$\fP$ covers all vertices of $\bX_\cU$;
\item
$\fP$ covers at most two vertices inside of each non-ugly cell;
\item
every vertex in $\bX$ that is covered by $\fP$ is at graph-distance at most $2(20d)^d$ from some vertex in $\bX_\cU$ with respect to the graph $\sG(\bX; \widehat r_{\delta\ge2})$;
\item
for each path $P\in\fP$, there is a good cell $C_P$ such that the two endvertices of $P$ lie in cells that are adjacent (in the graph of cells) to $C_P$;
\item
every two different paths in $\fP$ are at $\ell_p$-distance at least $Ar_0$ from each other.
\end{enumerate}
\end{lem}
\begin{proof}
Given any component $\cK$ of the graph $\sG_\cC[\cU]$ induced by the ugly cells, let $\bX_\cK$ denote the set of all vertices in $\bX$ contained in cells of $\cK$.
Assume that the a.a.s.\ conclusions of~\eqref{eq:Penrose}, \eqref{eq:r0rd2} and Lemmas~\ref{lem:ugly1}, \ref{lem:ugly2} and~\ref{lem:goodaroundugly} hold. Then, for each component $\cK$ of $\sG_\cC[\cU]$, we will deterministically find a path $P=P_\cK$ in $\sG(\bX; \widehat r_{\delta\ge2})$ such that:
\begin{enumerate}
\renewcommand{\theenumi}{\arabic{enumi}$'$}
\item
$P$ covers all vertices in $\bX_\cK$;
\item
$P$ covers at most two vertices inside of each non-ugly cell;
\item
$P$ only covers vertices within  graph-distance $2(20d)^d$ from $\bX_\cK$ with respect to $\sG(\bX; \widehat r_{\delta\ge2})$;
\item the two endvertices of $P$ are contained in cells that are both adjacent (in $\sG_\cC$) to the same good cell $C_P$.
\end{enumerate}
Set $A'= A+8(20d)^d$.
In view of Lemma~\ref{lem:ugly2}, for any two components $\cK$ and $\cK'$ of $\sG_\cC[\cU]$, the corresponding sets of vertices $\bX_\cK$ and $\bX_{\cK'}$ must be at distance at least $A'r_0$ from each other. Therefore, the paths $P_\cK$ and  $P_{\cK'}$ must be at distance at least $A'r_0-4(20d)^d\, \widehat r_{\delta\ge2} \ge A r_0$ from each other (since $\widehat r_{\delta\ge2} \le 2r_0$, by~\eqref{eq:r0rd2}). Combining this and properties~(1$'$--4$'$) above, the collection $\fP$ of all such paths $P_\cK$ will trivially satisfy all the conditions of the lemma. It only remains to prove the existence of these paths.

Pick any connected component $\cK$ of $\sG_\cC[\cU]$. If $\bX_\cK$ is empty, there is nothing to be done. Otherwise, the vertices in $\bX_\cK$ induce a clique in $\sG(\bX; r_0) \subseteq \sG(\bX; \widehat r_{\delta\ge2})$ (by Lemma~\ref{lem:ugly1}). Moreover, since $\sG(\bX; \widehat r_{\delta\ge2})$ is $2$-connected (by~\eqref{eq:Penrose}), we can find four different vertices $u,u'\in\bX_\cK$ and $v,v' \in \bX \setminus \bX_\cK$ (unless $|\bX_\cK|=1$, in which case we set $u=u'$) such that $uv$ and $u'v'$ are edges of $\sG(\bX; \widehat r_{\delta\ge2})$. Therefore, we can connect all vertices in $\bX_\cK$ by a path $P_0$ with endvertices $u$ and $u'$, and extend this path to a longer path $vP_0v'$  in $\sG(\bX; \widehat r_{\delta\ge2})$. Note that all edges in that path have length at most $r_0$ except for possibly $uv$ and $u'v'$, which have length at most $\widehat r_{\delta\ge2}$.
Let $D$ and $D'$ be respectively the cells containing $v$ and $v'$ (possibly $D=D'$). Note that $D$ and $D'$ may be good or bad cells but not ugly since $v,v' \in \bX \setminus \bX_\cK$, so in particular they must be adjacent in the graph of cells to some good cell. If $D=D'$, then the path $P=vP_0v'$ already satisfies properties~(1$'$--4$'$), with $C_P$ being any good cell adjacent to $D=D'$.
Similarly, if $D$ and $D'$ are adjacent in $\sG_\cC$ and $D'$ is a good cell, then we pick any vertex $v''\ne v'$ in cell $D'$ (it must exist, since $D'$ is dense), and set $P=v''vP_0v'$. The obtained path also meets our requirements, with $C_P$ being any good cell adjacent to $D'$. The symmetric case can be dealt with analogously. Thus, we can restrict ourselves to the case in which $D\ne D'$ and moreover $D$ and $D'$ are both bad or non-adjacent.  
In particular, each one of them is adjacent to some good cell not in $\{D,D'\}$. Under these assumptions, we will extend $vP_0v'$ to a longer path with the desired properties, by adding at most $2(20d)^d-1$ extra edges of length at most $r_0$.

Fix any arbitrary ugly cell $C\in\cK$, and let $\cG'$ be the set of good cells at $\ell_\infty$-distance at most $3r_0$ from $C$.
By Lemma~\ref{lem:ugly1}, the union of all cells in $\cK$ has $\ell_\infty$-diameter at most $4d^2s$.
Therefore, $D$ and $D'$ must be at $\ell_\infty$-distance at most $\widehat r_{\delta\ge 2}+4d^2s \le 3r_0/2$ from cell $C$ (by~\eqref{eq:r0rd2} and our choice of $\epsilon$ sufficiently small).
From our assumptions, each of $D$ and $D'$ must be adjacent in $\sG_\cC$ to some good cell (not in $\{D,D'\}$) at $\ell_\infty$-distance at most $3r_0/2 + s + (r_0-2ds) \le 3r_0$ from $C$. That is, $D$ and $D'$ must be adjacent to some cell in $\cG' \setminus \{D,D'\}$. In view of Lemma~\ref{lem:goodaroundugly}, there exists a path of cells $D',D_1,D_2,\ldots,D_jD$ in the graph $\sG_\cC$ such that cells $D_1,D_2,\ldots,D_j$ belong to $\cG'\setminus \{D,D'\}$ and $1\le j\le 2(20d)^d$. Choose one vertex $v_i\in\bX$ inside each cell $D_i$ ($1\le i\le j$). Then, the path $P = vP_0v'v_1v_2\cdots v_{j-1}$ satisfies our desired properties (with cell $C_P=D_j$).
\end{proof}
Finally, we provide a (simpler) analogue of Lemma~\ref{lem:paths} that will be used for the construction of the rainbow perfect matching.
\begin{lem}\label{lem:paths2}
Let $\bX_\cU$ denote the set of vertices in $\bX$ that belong to ugly cells. Given any constant $A>0$, a.a.s.\ there is a collection $\fP$ of vertex-disjoint paths in $\sG(\bX; \widehat r_{\delta\ge1})$ such that:
\begin{enumerate}
\renewcommand{\theenumi}{\arabic{enumi}$''$}
\item
$\fP$ covers all vertices of $\bX_\cU$;
\item
every path in $\fP$ has an even number of vertices, and contains at most one vertex not in $\bX_\cU$;
\item
every two different paths in $\fP$ are at $\ell_p$-distance at least $Ar_0$ from each other.
\end{enumerate}
\end{lem}
\begin{proof}[Proof (sketch).]
We proceed similarly as in the proof of Lemma~\ref{lem:paths}, so we just sketch the main differences. For each component $\cK$ of $\sG_\cC[\cU]$,
the vertices in $\bX_\cK$ induce a clique in $\sG(\bX; r_0) \subseteq \sG(\bX; \widehat r_{\delta\ge1})$ (by Lemma~\ref{lem:ugly1}).
If the number of vertices in $\bX_\cK$ is even, we simply connect them all by a path $P_\cK$. If it is odd, we use the fact that $\sG(\bX; \widehat r_{\delta\ge1})$ is a.a.s.\ $1$-connected (by~\eqref{eq:Penrose}) to find vertices $u\in\bX_\cK$ and $v\in \bX \setminus \bX_\cK$ at distance at most $\widehat r_{\delta\ge1}$ from each other (i.e.~$uv$ is an edge of $\sG(\bX; \widehat r_{\delta\ge1})$). Then we pick a path that connects all vertices in $\bX_\cK$ and has $u$ as an endpoint, and extend it to $P_\cK$ by adding edge $uv$.
\end{proof}

\section{Finding rainbow paths and cycles}\label{sec:localtasks}

In this section, we will build a rainbow spanning graph of $\sG(\bX; \bZ; \widehat r_{\delta\ge k})$, for $k\in\{1,2\}$,  consisting of small (path and cycle) components, which will be later used in the construction of a rainbow perfect matching ($k=1$) or a rainbow Hamilton cycle ($k=2$).
We will proceed in a greedy fashion, and build the rainbow small pieces in a specific order, since the calculations will only work if we reveal the colours of certain edges before others.

At this stage, we expose all the points of $\bX$ (which determine which cells are ugly, bad and good), and
assume henceforth that the a.a.s.\ conclusions of all statements in Section~\ref{sec:tess} hold. In particular, all probabilistic statements in the sequel will refer only to the random assignment $\bZ=(Z_{ij})_{1\le i<j\le n}$ of colours to edges.
In the case that our final goal is to build a rainbow Hamilton cycle in $\sG(\bX; \bZ; \widehat r_{\delta\ge 2})$, we pick a collection $\fP$ of vertex-disjoint paths satisfying properties~1--5 in Lemma~\ref{lem:paths} with $A=3$ (or $A=1000$ for that matter, as we only need to guarantee that paths in $\fP$ are far enough from each other).
Otherwise, in order to obtain a rainbow perfect matching in $\sG(\bX; \bZ; \widehat r_{\delta\ge 1})$, we pick a collection $\fP$ of vertex-disjoint paths
satisfying properties~1$'$--3$'$ in Lemma~\ref{lem:paths2}. Note that in the latter case, since $\widehat r_{\delta\ge 1} \le \widehat r_{\delta\ge 2}$, all the edges in the paths of $\fP$ also belong to $\sG(\bX; \widehat r_{\delta\ge2})$ and moreover $\fP$ satisfies conditions 1, 2, 3 and 5 (but not necessarily 4) in Lemma~\ref{lem:paths}.
These are in fact all the assumptions on $\fP$ that we will need in this section.

%
\begin{defn}
We will refer to the paths in $\fP$ as {\em ugly paths} (since they cover all the vertices in ugly cells). 
Let $\bX'$ be the set of all vertices in ugly paths, and  let $\bX'' = \bX \setminus \bX'$.
\end{defn}

Note that non-ugly (i.e.~good or bad) cells may contain both vertices of $\bX'$ and vertices of $\bX''$. However, by property~2 of Lemma~\ref{lem:paths}, each of these cells contains at most two vertices of $\bX'$.
In view of this and since bad cells are sparse and good cells are dense, we conclude that each bad cell contains at most $\eps^3\log n$ vertices of $\bX''$, and each good cell contains at least $\eps^3\log n-2$ vertices of $\bX''$.
We will first analyze the colours of the edges in the ugly paths, and then ``delete'' the corresponding vertices (i.e.~$\bX'$) from all bad and good cells in $\cC$.
A high-level description of our construction can be summarized in the following steps.
\begin{enumerate}
\item
First we reveal the colours of the edges in the paths of $\fP$ (i.e.~the ugly paths).
We will show that a.a.s.\ we do not get any repeated colours, so the ugly paths form a rainbow forest.
\item
Next we consider bad cells one by one. In each bad cell, we fix some arbitrary Hamilton path on the vertices of $\bX''$ contained in that cell (ignore those in $\bX'$), reveal the colours of all the edges on this path, and then greedily discard edges with previously used colours in the process (i.e.~in that same or previous bad cells or in ugly paths). We will see that a.a.s., only at most a few edges are removed from the path, leaving behind a small number of paths. These paths together with the ugly ones are all rainbow by construction. 
\item
Next we run through good cells one by one (we again restrict our attention to vertices of $\bX''$ in those cells and ignore those in $\bX'$). Within a particular good cell, we run through the vertices revealing the colours of the edges to other vertices in the same cell. We will construct a rainbow Hamilton cycle within each good cell, and a.a.s.\ these cycles will have no colour collisions with each other or any of the previously constructed parts. 

\item
Finally we hook up all the parts to build either a rainbow perfect matching or a rainbow Hamilton cycle.
This requires adding some extra edges that connect vertices in different cells (one of which is always good).
Fortunately, we have plenty of these edges available and many unused colours, so we can a.a.s.\ find the required edges with no colour collisions.
\end{enumerate}

In the remaining of the section, we will focus on steps 1--3, which produce a rainbow collection of vertex-disjoint paths and cycles covering all vertices in $\bX$.
The final step is described in Section~\ref{sec:connect}.

\subsection{Ugly paths}\label{ssec:ugly}

We will show that the number of edges in ugly paths is so small that we do not expect colour collisions among them, so in particular the collection of ugly paths is a rainbow forest.
Given a graph $G$, let $E(G)$ denote the edge set of $G$.
\begin{lem}\label{lem:ugly3}
Assume that $\bX$ satisfies all the a.a.s.\ statements in Section~\ref{sec:tess}, and pick a collection  $\fP$ of (ugly) paths satisfying conditions 1,2,3,5 of Lemma~\ref{lem:paths}.
Then a.a.s.\ $\bigcup_{P\in\fP} E(P)$ is rainbow with respect to the random edge colouring $\bZ$.
\end{lem}
\begin{proof}
Recall that $\bX_\cU$ is the set of vertices in ugly cells, and $\bX'$ is the set of vertices in ugly paths.
Let $\widehat\bX_\cU$ be the set of vertices in $\bX$ within graph-distance $2(20d)^d$ from some vertex in $\bX_\cU$ with respect to the graph $\sG(\bX; \widehat r_{\delta\ge2})$. By condition~3 of Lemma~\ref{lem:paths}, we have $\bX' \subseteq \widehat\bX_\cU$.
We wish to obtain a bound on $|\bX'|$.
Note that $|\bX_\cU| \le n^{0.4}$ by Lemmas~\ref{lem:maxdensity} and~\ref{lem:badugly} (assuming $\eps$ is sufficiently small).
Moreover, by Lemma~\ref{lem:DeltaRGG}, every vertex of $\bX_\cU$ has at most $O(\log n)$ vertices within graph-distance $\ell=2(20d)^d$ in $\sG(\bX; \widehat r_{\delta\ge2}) \subseteq \sG(\bX,r_1)$ (see also~\eqref{eq:r0rd2}). Therefore,
\begin{equation}\label{eq:sizeXprime}
|\bX'| \le |\widehat\bX_\cU| =O(n^{0.4} \log n),
\end{equation}
and the total number of edges in ugly paths is at most $|\bX'|-1$. Then, the probability that no colour is repeated across these edges is at least
\[
\left( 1- \frac { |\bX'| }{Kn} \right)^{|\bX'|} = \exp \left( O\left( \frac { (n^{0.4} \log n)^2}{Kn} \right) \right) = 1-o(1).
\]
So in particular, a.a.s.\ all the edges in ugly paths receive distinct colours. 
\end{proof}
Therefore, we can expose all the colours of the edges in the ugly paths and save them for future use. In other words, we are entitled to use any of these edges for our rainbow perfect matching or Hamilton cycle, but cannot use any of their colours anywhere else.

\subsection{Bad cells}\label{ssec:bad}

In this section, we restrict our atention to vertices in $\bX''$ contained in bad cells (and ignore $\bX'$). Recall that the set of vertices in any cell induces a clique in $\sG(\bX''; r_0)$, so we may use any possible edge between two vertices in a cell. Our goal is to show that a.a.s.\ every bad cell contains a rainbow spanning linear forest with at most $4/\eps$ many path components.
Further, the colours used across all of these forests are distinct and also separate from the colours used on the paths in $\fP$.

\begin{lem}\label{lem:bad}
Under the same assumptions as in Lemma~\ref{lem:ugly3}, a.a.s.\ for every bad cell $C\in\cB$ the complete graph on the vertices of $\bX''$ inside of $C$ contains a spanning forest $F_C$ consisting of at most $4/\eps$ many paths (possibly isolated vertices) such that $\bigcup_{C\in\cB}E(F_C) \cup \bigcup_{P\in \fP}E(P)$ is rainbow.
\end{lem}

\begin{proof}
By Lemma \ref{lem:badugly}, we may assume that there are at most $n^{1-\eps/2}$ many bad cells.  We now proceed greedily, moving through all the bad cells one by one. In a given bad cell $C$, we ignore any vertices which belong to a path of $\fP$ (there are at most 2 such vertices), and connect the remaining vertices by any arbitrary Hamilton path $P_C$. Then, delete any edge from this path which receives a colour used already in $P_C$ or in a previous bad cell or in a path of $\fP$.
Call the resulting graph $F_C$, which is a spanning linear forest by construction and satisfies the required rainbow conditions. It only remains to prove that we did not delete too many edges from path $P_C$, so that $F_C$ has at most $4/\eps$ components.
Thus far, we have used at most $O(n^{1 - \eps/2}\log n)$ colours: at most $O(n^{0.4}\log n)$ on  paths of $\fP$ (by~\eqref{eq:sizeXprime}) and at most $\eps^3\log n \le \log n$ other colours in each of the previous bad cells.
So the probability that an edge is deleted from $P_C$ is $O( n^{1 - \eps/2}\log n / (Kn)) = O(\log n / n^{\eps/2})$. Since these events are independent, the probability of deleting more than $3/\eps$ edges in a cell is at most
$$
\binom{ \lceil\log n\rceil }{ \lceil 3/\eps \rceil } \cdot \left( O\left( \frac {\log n}{n^{\eps/2}} \right) \right)^{ 3/\eps } 
\le n^{-3/2+o(1)} = o(1/n).
$$
Hence, a.a.s.\ it will not happen for any cell in $\cB$ in the greedy process.
\end{proof}

As we already mentioned, we keep these paths and forbid their colours for future use.

\subsection{Good cells}\label{ssec:good}
In this section we will prove the following lemma which shows that each good cell contains a rainbow Hamilton cycle which does not use any previously used colours. As in Section~\ref{ssec:bad}, we restrict our attention to the vertices of $\bX''$ contained in each cell.
\begin{lem}\label{lem:good}
Under the same assumptions as in Lemma~\ref{lem:ugly3}, a.a.s.\ for every good cell $C\in\cG$ the complete graph on the vertices of $\bX''$ inside of $C$ contains a Hamilton cycle $H_C$ such that $\bigcup_{C\in \cG}E(H_C)\cup\bigcup_{C\in\cB}E(F_C) \cup \bigcup_{P\in \fP}E(P)$ is rainbow (where $\bigcup_{C\in\cB}E(F_C)$ is the linear forest obtained in Lemma~\ref{lem:bad}).
\end{lem}

\begin{proof}
We visit all good cells one at a time. Inside of each good cell, we build a rainbow Hamilton cycle restricted to vertices of $\bX''$ only (and where each pair of vertices is regarded as a potential edge). Good cells are dense, so each cell contains at least $\eps^3\log n -2$ vertices of $\bX''$ (by condition~2 of Lemma~\ref{lem:paths}).
More specifically, we do the following within each good cell:
we examine each edge one at a time, revealing its colour. We keep an edge if
its colour: (i) has not been used on $\bigcup_{C\in\cB}E(F_C)$ or $\bigcup_{P\in\fP}E(P)$, (ii)  has not been used in a rainbow Hamilton cycle from a previous good cell, and (iii) has not been seen previously in this cell.
Otherwise, we delete the edge.
At each step, the number of unusable colours for an edge is at most $n+o(n)$ (at most $n$ colours are referred to in (ii) and we may add $o(n)$ to account for (i) and (iii)).  
Hence, an edge is not present within a good cell with probability at most $(1+o(1))/K$. Suppose a good cell has $x \ge \eps^3 \log n  -2$ vertices of $\bX''$ in it. The probability that, when we reveal the edges incident to some fixed vertex, we see that $\lceil x /2\rceil$ of them are unusable is at most 
\[
\binom{x}{\lceil x /2\rceil} \bfrac{1+o(1)}{K}^{x /2} \le 2^x \bfrac{1+o(1)}{K}^{x /2} \le (5/K)^{(\eps^3/2) \log n}  = o (1/n)
\] 
so long as we choose our constants such that $(\eps^3/2) \log (K/5) > 1$. (Note that earlier arguments required $\eps$ to be sufficiently small, but $K$ can be chosen large enough with respect to this $\eps$.)
Thus, by the union bound over all vertices of $\bX''$ in dense cells, we conclude that a.a.s.\ each good cell contains a rainbow Dirac graph (a graph on $s$ vertices  with minimum degree at least $s/2$). Such graphs are Hamiltonian~\cite{Dir52} and so a.a.s.\ each good cell contains a rainbow Hamilton cycle.  Moreover, by construction, any colour used in such a cycle is not used in any other such cycle, nor in any path or forest constructed in ugly and bad cells. 
\end{proof}

\section{Connecting the good, the bad and the ugly}\label{sec:connect}
In this section, we complete the proof of Theorem~\ref{thm:main}. We will first show that a.a.s.\ $\sG(\bX;\bZ;\widehat r_{\delta\ge2})$ contains a rainbow Hamilton cycle, and then adapt the argument for the corresponding statement about a rainbow perfect matching.

We assume hereafter that $\bX$ is fixed and satisfies all the a.a.s.\ statements in Section~\ref{sec:tess}. Moreover, we pick a collection $\fP$ of ugly paths that meets all the requirements in Lemma~\ref{lem:paths} (with $A=3$), and assume that the a.a.s.\ conclusions of Lemmas~\ref{lem:ugly3}, \ref{lem:bad} and~\ref{lem:good} hold. This implies that we have a rainbow graph with edge set $\bigcup_{C\in \cG}E(H_C)\cup\bigcup_{C\in\cB}E(F_C) \cup \bigcup_{P\in \fP}E(P)$ that covers all vertices of $\bX$ and is made of path and cycle components. Furthermore, all edges in the ugly paths are of length at most $\widehat r_{\delta\ge2}$, whilst the remaining ones are of length at most $ds\le r_0$ (since they join pairs of vertices inside the same cell), so in particular the above graph is a rainbow subgraph of $\sG(\bX;\bZ;\widehat r_{\delta\ge2})$. We will obtain a rainbow Hamilton cycle by adding a few extra edges (and deleting some others accordingly) that join together the cycle and path components and preserve the rainbow condition. These new edges will be chosen so that their endpoints lie in different but adjacent cells in the graph of cells $\sG_\cC$ (so they have length at most $r_0$), and at least one of their endpoints is in $\bX''$ (i.e.~not in an ugly path). In particular, the colours of these new edges have never been revealed during the greedy process that lead us to Lemmas~\ref{lem:ugly3}, \ref{lem:bad} and~\ref{lem:good}, and thus remain random.

Pick a spanning tree $T$ of $\sG_\cC[\cG]$ (the large component induced by the good cells in the graph of cells).
Then its maximum degree satisfies $\Delta(T) \le \Delta(\sG_\cC) = O(1/\eps)$. 
Now by definition, each bad cell is adjacent (in $\sG_\cC$) to some cell in $\cG$.
Thus we may define a tree $T'$ on vertex set $\cG\cup\cB$ by connecting each bad cell to one of the good cells it is adjacent to in $\sG_\cC$. Then the cells of $\cB$ appear as leaves in $T'$.
Again we have the bound $\Delta(T') = O(1/\eps)$ since $T'$ is a subgraph of $\sG_\cC$. %
Finally, consider the collection of ugly paths $\fP$ we picked from Lemma \ref{lem:paths}. Recall that each path $P\in \fP$ has a corresponding good cell $C_P$. We define tree $T''$ by adding each $P\in\fP$ as new a leaf of $T'$ attached to good cell $C_P$.
Since paths in $\fP$ are far apart (by property~5 of Lemma~\ref{lem:paths}), no good cell has more than one such pendant edge attached and so we again have $\Delta(T'') = O(1/\eps)$.

We now use $T''$ as a template to create the rainbow Hamilton cycle.  For any ugly path $P$, we need to find an edge in the cycle $H_{C_P}$ of the good cell $C_P$ adjacent to $P$ (in $T''$), remove that edge and attach path $P$ to the endpoints of that edge. This extends $H_{C_P}$ to a larger cycle that covers $P$.  
Similarly, for each bad cell $C\in\cB$, we will attach each of the paths of $F_C$ to the cycle $H_{C'}$ of the good cell $C'$ that is adjacent to $C$ in $T''$.
We refer to these paths as {\em bad paths}.
Note that in this operation each bad path of $F_C$ is attached to the endpoints of a different edge of $H_{C'}$, which is immediately deleted.
This uses at most $4/\eps$ edges of $H_{C'}$ (by Lemma~\ref{lem:bad}).
Finally, we will attach the cycles within adjacent (in $T''$) good cells to each other. This operation requires deleting one edge in each cycle and joining the endpoints of the resulting paths together so that we obtain one longer cycle.
Iterating this procedure for all edges of $T''$ produces a Hamilton cycle in $\sG(\bX;\widehat r_{\delta\ge2})$ provided that we have enough edges in the cycles of the good cells.

It remains to show that we can a.a.s.\ connect all the pieces together in the manner described in the above paragraph and at the same time satisfy the rainbow condition. The potentially most delicate steps in our construction are when we hook up ugly/bad paths to cycles inside good cells. Consider such an ugly/bad path
$P$ with endpoints $u,v$ contained in cell $C$, and let $C'$ be the good cell adjacent to $C$ in $T''$.
We will show that we can find some edge $u'v'$ in the cycle $H_{C'}$ of the good cell such that the new edges $uu'$ and $vv'$ connecting path $P$ to cycle $H_{C'}$ are assigned previously unused colours. We call $u'v'$ the {\em hook edge} for the ugly/bad path $P$. Then, the hook edge $u'v'$ can be removed and replaced by the ugly/bad path together with the new edges $uu'$ and $vv'$.
Recall that, since cell $C'$ is good, it contains $x \ge \eps^3 \log n-2$ vertices of $\bX''$, and therefore cycle $H_{C'}$ has $x$ potential hook edges.
To ensure independence, we only consider every second edge in the cycle. Moreover, at most $O(1/\eps^2)$ edges are perhaps already used as hook edges to attach some other ugly/bad paths to cycle $H_{C'}$, (since the degree of $T''$ is $O(1/\eps)$ and each bad cell contains at most $4/\eps$ bad paths). In any case, there are at least $\eps^3 \log n / 2 -O(1/\eps^2) \ge \eps^3 \log n / 3$ edges in the cycle we can still use. The probability that no such edge has the property we seek (i.e.\ that the edges connecting the endpoints of the edge to the endpoints of the ugly/bad path are not both of unused colours) is at most
\[
\left( 1 - \left( \frac {K-1+o(1)}{K} \right)^2 \right)^{\eps^3 \log n / 3} \le \left( \frac {2}{K} \right)^{\eps^3 \log n / 3} = \exp \left( - \frac {\eps^3 \log(K/2)}{3} \log n \right) =o \bfrac{1}{n} 
\]
so long as we ensure that $K$ is large enough so that $\log(K/2) \eps^3  / 3 >1$.
Recall that we have $o(n)$ ugly/bad paths in total. Hence, a union bound over all the ugly/bad paths implies that we will a.a.s\  succeed at finding a hook edge for each such path. The argument to merge cycles $H_C$ and $H_{C'}$ of two good cells $C$ and $C'$ together is similar. This time we need to find two hook edges $uv$ in $H_C$ and $u'v'$ in $H_{C'}$ such that the new edges $uu'$ and $vv'$ receive previously unused colours. We have $\Omega_\eps(\log^2 n)$ choices of pairs of hook edges and so the failure probability is at most $\of{\frac{2}{K}}^{\Omega_\eps(\log^2 n)} =o\bfrac{1}{n}$.
This completes the proof of the first part of Theorem~\ref{thm:main}.

Finally, we will show that a.a.s.\ $\sG(\bX; \widehat r_{\delta\ge1})$ contains a perfect matching (for even $n$).
The argument will reuse most of the ideas in the construction of a rainbow Hamilton cycle earlier in this section, so we will only sketch the main differences.
This time, we will assume that $\fP$ satisfies conditions (1$''$--3$''$) from Lemma~\ref{lem:paths2}, and deterministically build a perfect matching.
Recall that our new assumptions on $\fP$ also imply that conditions 1,2,3,5 (but not necessarily 4) in Lemma~\ref{lem:paths} are true, so we are entitled to assume all the claims in Section~\ref{sec:localtasks}. Additionally, the ugly paths in $\fP$ have an even number of vertices and (by construction) only use edges of length at most $\widehat r_{\delta\ge 1}$.

We proceed as before but using $T'$ instead of $T''$ as a template and thus ignoring ugly paths. That is, we hook up each bad path to the cycles in the corresponding good cell and also the cycles within any two good cells that are adjacent in $T'$. However, we do not attach ugly paths to anything (in fact, we may not be able to do so, since $\fP$ may not satisfy property~4 in Lemma~\ref{lem:paths}). This procedure a.a.s.\ creates a big rainbow cycle $H$ in $\sG(\bX; \bZ; r_0) \subseteq \sG(\bX; \bZ; \widehat r_{\delta\ge1})$ that covers all vertices except for those in ugly paths, and moreover $H\cup\bigcup_{P\in\fP}P$ is rainbow.
If we restrict asymptotics to even $n$, then cycle $H$ has even length since all paths of $\fP$ have an even number of vertices. 
By removing alternating edges adequately from $H$ and the ugly paths, we obtain a rainbow perfect matching in $\sG(\bX; \widehat r_{\delta\ge1})$ as desired.
This implies the second statement of Theorem~\ref{thm:main}, and finishes the proof.

\section{Case $\mbf{p=1}$}\label{sec:p1}

In this section, we consider the case $p=1$, and sketch how to adapt the argument of Theorem~\ref{thm:main} in order to obtain~\eqref{eq:p1}. Recall that for $p=1$, it is not known whether or not~\eqref{eq:Penrose} holds (this is due to some technical parts of the argument in~\cite{Pen99} that break down for $p=1$). However, we can still claim that, for $p=1$ and any constant $\eta>0$,
\[
(1-\eta)r_0 \le \widehat r_{\1conn} \le \widehat r_{\2conn} \le (1+\eta)r_1 \qquad \text{a.a.s.}
\]
This follows from the fact that~\eqref{eq:Penrose} and~\eqref{eq:r0rd2} are valid for every $p>1$ and by continuity of $\theta=\theta(p)$ at $p\ge 1$.
Then we can pick $\eta$ sufficiently small, and replace $r_0$, $\widehat r_{\delta\ge1}$, $\widehat r_{\delta\ge2}$ and $r_1$ by $(1-\eta)r_0$, $\widehat r_{\1conn}$, $\widehat r_{\2conn}$ and $(1+\eta)r_1$, respectively, in the proof of Theorem~\ref{thm:main}. The argument is still valid with virtually no adaptation, and yields~\eqref{eq:p1}.

\section{Open questions}\label{sec:future}


In this paper, we showed that a.a.s.\ 
the first edge in the edge-coloured random geometric graph process $\big(\sG(\bX;\bZ;r)\big)_{r\ge0}$ that gives minimum degree at least $2$ (or $1$) and such that at least $n$ (or $n/2$) colours have appeared also creates a
rainbow Hamilton cycle (or perfect matching), provided that the number of colours is at least $c=\lceil Kn\rceil$, where $K=K(d)>0$ is a sufficiently large constant. This condition on $c$ ensures that at least $n$ (or $n/2$) colours have appeared long before the minimum degree becomes $2$ (or $1$).

Thus the most intriguing  open question is to prove that these statements hold for any number of colours, $c$. Of course for Hamilton cycles, we must have $c\ge n$ and for perfect matchings, we must have $c\ge n/2$. 
This problem may be particularly interesting in the case of perfect matchings when $d=2$ and $c=n/2$. In this case, the first appearance of a perfect matching and of $n/2$ distinct colours occurs once $(1+o(1)) \frac{n}{2} \log n$ many edges have arrived. The case of Hamilton cycles when $d=4$ and $c=n$ is interesting for the analogous reason. 
The first results on packing rainbow Hamilton cycles (that is, finding a collection of edge disjoint Hamilton cycles) in $G_c(n,p)$ were recently obtained in~\cite{FKMS14}.
We believe that using some of the ideas in~\cite{MPW11}, it should be relatively easy to extend our argument to find a constant number of edge-disjoint rainbow Hamilton cycles and perfect matchings in $\sG(\bX;\bZ;r)$ as well.
It would be interesting to consider further extensions in which the number of rainbow Hamilton cycles or perfect matchings in the packing grows to infinity as a function of $n$.

\bibliographystyle{abbrv}
\bibliography{rgg-rainbow}

\end{document}